\documentclass[11pt,a4paper]{amsart}
\usepackage{amsmath,amssymb,color,multicol,setspace}
\usepackage[utf8,utf8x]{inputenc}

\usepackage[pdftitle = {The\ Plesken\ Lie\ algebra\ for\ associative \ algebras\ with\ an\ anti-involution}]{hyperref}
\usepackage{longtable}
\usepackage{xy,amscd}
\usepackage{epsfig}
\usepackage{rotating}
\usepackage{xspace}
\usepackage{enumitem}
\xyoption{all}
\usepackage{tikz}
\usetikzlibrary{arrows,decorations.pathmorphing,decorations.pathreplacing,positioning,shapes.geometric,shapes.misc,decorations.markings,decorations.fractals,calc,patterns}
\usepackage{xcolor}
\usepackage{comment}

\newtheorem{Lemma}{Lemma}[section]
\newtheorem{Theorem}[Lemma]{Theorem}
\newtheorem{Proposition}[Lemma]{Proposition}
\newtheorem{Corollary}[Lemma]{Corollary}

\theoremstyle{definition}
\newtheorem{Definition}[Lemma]{Definition}

\newtheorem{Remark}[Lemma]{Remark}

\newtheorem{Example}[Lemma]{Example}

\numberwithin{equation}{section}

\title[Plesken Lie algebras for associative algebras with anti-involution]{The Plesken Lie algebra for associative algebras with anti-involution: semisimple cellular algebras}

\author{Thorsten~Holm}
\address{Thorsten Holm, Leibniz Universit\"at Hannover,
Institut f\"ur Algebra, Zahlentheorie und Diskrete Mathematik,
Fakult\"at f\"ur Mathematik und Physik,
Welfengarten 1,
D-30167 Hannover, Germany}
\email{holm@math.uni-hannover.de}
\urladdr{https://www.iazd.uni-hannover.de/de/holm}

\author{Nils~Wirries}
\address{Nils Wirries, Leibniz Universit\"at Hannover,
Institut f\"ur Algebra, Zahlentheorie und Diskrete Mathematik,
Fakult\"at f\"ur Mathematik und Physik,
Welfengarten 1,
D-30167 Hannover, Germany}
\email{nils.wirries@stud.uni-hannover.de}



\begin{document}

\begin{abstract}
Cohen and Taylor, following an idea of Plesken, introduced a Lie algebra to the complex group algebra of a finite group
and determined its structure, based on the character theory of the group. We show how the definition of
this Plesken Lie algebra can be extended to any associative algebra with an anti-involution. After some
examples we consider semisimple cellular algebras and prove that their Plesken Lie algebras are direct sums
of orthogonal Lie algebras, the sizes of which are determined by the dimensions of the cell modules 
of the cellular algebra.   
\end{abstract}

\maketitle

\section{Introduction}

In their paper \cite{CT07}, Cohen and Taylor presented a surprising construction of a Lie algebra
which can be associated to the complex group algebra of a finite group, see also \cite{M10}.
They attribute the idea to Plesken, and hence these Lie algebras are now called Plesken Lie algebras.  
The structure of the Plesken Lie algebra of a finite group has been completely determined by Cohen and Taylor,
it is a direct sum of orthogonal, symplectic and general linear Lie algebras, where the summands 
depend on the ordinary character theory of the group, namely on the sets of irreducible characters
of real type, symplectic type and complex type (see \cite[Theorem 5.1]{CT07}).  We review this result
in Section \ref{sec:Pleskengroup}. 

The Plesken
Lie algebra of a group $G$ is the subspace of $\mathbb{C}G$ spanned by the elements $g-g^{-1}$ 
with $g\in G$. 
In Section \ref{sec:Pleskeninvolution} we show how this construction can be generalized to 
any associative algebra with an anti-involution (this has already been mentioned in \cite[Section 1.1]{M10}, but 
this paper then only deals with group algebras). In Section \ref{sec:examples} we present some 
small examples of Plesken Lie algebras for algebras other than group algebras, namely for quaternions and for 
matrix algebras over complex numbers with anti-involutions given by conjugation 
of numbers and transposition of matrices. 

A large and important class of associative algebras which naturally carry an anti-involution are the 
cellular algebras introduced by Graham and Lehrer \cite{GL96}, see also \cite{KX98} for an alternative 
definition. A crucial feature of cellular algebras is the existence of a special type of basis with 
particular properties, which are motivated by the Kazhdan-Lusztig bases of Hecke algebras. 
We recall the precise definition of a cellular algebra in Section \ref{sec:Pleskencellular}. 
Following Graham and Lehrers seminal article, there has been an extensive literature on 
cellular algebras. Many algebras which appear naturally in combinatorial representation 
theory and mathematical physics have been proven to be cellular, e.g. Ariki-Koike Hecke algebras \cite{GL96},
Hecke algebras for finite Weyl groups \cite{G07}, $q$-Schur algebras of finite type
\cite{CLX23}, Khovanov-Lauda-Rouquier (KLR) algebras \cite{MT23}, 
partition algebras \cite{Xi99}, Brauer algebras \cite{GL96}, Temperley-Lieb algebras 
\cite{GL96}, Jones algebras \cite{GL96}, and many others. 

As our main result we determine the structure of the Plesken Lie algebras for all semisimple 
cellular algebras, see Theorem \ref{thm:Plesken_cellular}. The semisimplicity assumption here is
not too restrictive as cellular algebras are often generically semisimple.
We show that these Plesken
Lie algebras are direct sums of orthogonal Lie algebras, whose sizes are given by the 
dimensions of the (non-zero) cell modules of the cellular algebra. 

In Section \ref{sec:diagram} we illustrate our main result by considering two classes of
diagram algebras, the planar rook algebras and the Temperley-Lieb algebras. All the 
planar rook algebras $PR(n)$ are semisimple and Corollary \ref{cor:planarrook} 
gives the structure of the Plesken Lie algebra of $PR(n)$ as direct sum of orthogonal 
Lie algebras. The semisimplicity of the Temperley-Lieb algebras $TL_{\delta}(n)$  
depends on the parameter $\delta$, in Corollary \ref{cor:PleskenTL} we give 
the structure of the Plesken Lie algebras of $TL_{\delta}(n)$ for a large
set of parameters. We end the section with the explicit example of the
non-semisimple Temperley-Lieb algebra $TL_0(4)$ and show that its Plesken
Lie algebra can indeed not be a direct sum of orthogonal Lie algebras. 

The results of Section \ref{sec:diagram} are part of the master thesis of the
second author \cite{W25}, where he determined the structure of the Plesken 
Lie algebras for planar rook algebras and Temperley-Lieb algebras directly, 
without using the machinery of cellular algebras. The ideas for writing this article then
grew out of further discussions following the master thesis.

\section{The Plesken Lie algebra of a finite group} \label{sec:Pleskengroup}

In \cite{CT07}, Cohen and Taylor study a certain Lie algebra associated to any finite group. 
The authors attribute the definition of this Lie algebra to Plesken, and hence this Lie algebra is
now often called the Plesken Lie algebra of the group. 
We first recall the relevant definitions from \cite{CT07} and then go on to generalize 
the Plesken Lie algebra to other associative algebras. 

Let $G$ be a finite group. The complex group algebra $\mathbb{C}G$ 
can be made into a Lie algebra over $\mathbb{C}$ by setting the Lie bracket to be 
$[a,b]=ab-ba$ for every $a,b\in \mathbb{C}G$. 

The {\em Plesken Lie algebra} of the finite group $G$ is the subspace
$$\mathcal{L}(G) := \mathrm{span}\{g-g^{-1}\,|\,g\in G\}.
$$
We check that this is indeed a Lie algebra (with the Lie bracket induced from $\mathbb{C}G$). 
For any group element $g\in G$ we set $\widehat{g}:=g-g^{-1}$. For the Lie bracket we then have 
\begin{eqnarray*}
[\widehat{g},\widehat{h}] & = & \widehat{g}\,\widehat{h} - \widehat{h}\,\widehat{g}
= (g-g^{-1})(h-h^{-1})-(h-h^{-1})(g-g^{-1}) \\
& = & gh-gh^{-1}-g^{-1}h+g^{-1}h^{-1} -hg+hg^{-1}+h^{-1}g-h^{-1}g^{-1} \\
& = & \widehat{gh} - \widehat{gh^{-1}} -\widehat{g^{-1}h} +\widehat{g^{-1}h^{-1}} \in \mathcal{L}(G).
\end{eqnarray*}
Hence the subspace $\mathcal{L}(G)=\{\widehat{g}\,|\,g\in G\}$ 
is closed under the Lie bracket and is therefore itself a
Lie algebra. 

Cohen and Taylor \cite{CT07} completely determine the structure of the Lie algebras $\mathcal{L}(G)$
in terms of the character theory of the group $G$. 

\begin{Theorem}[\cite{CT07}, Theorem 5.1]
Let $G$ be a finite group. The Plesken Lie algebra $\mathcal{L}(\mathbb{C}G)$ admits the decomposition
$$\mathcal{L}(\mathbb{C}G) \cong \bigoplus_{\chi\in \mathfrak{R}}\,\mathfrak{o}(\chi(1),\mathbb{C})
\oplus \bigoplus_{\chi\in \mathfrak{Sp}}\,\mathfrak{sp}(\chi(1),\mathbb{C}) \oplus
 \bigoplus_{\chi\in \mathfrak{C}}\, `\mathfrak{gl}(\chi(1),\mathbb{C})
$$
where $\mathfrak{R}$, $\mathfrak{Sp}$, and $\mathfrak{C}$ are the sets of irreducible characters of $G$ of
real, symplectic and complex types, and where the prime signifies that there is just one summand 
$\mathfrak{gl}(\chi(1),\mathbb{C})$ for each pair $\{\chi,\overline{\chi}\}$ from $\mathfrak{C}$.
\end{Theorem}

Moreover, Cohen and Taylor draw several consequences from this structural result, for instance, they 
characterize for which groups the Plesken Lie algebra is a simple Lie algebra: with two exceptions this 
happens if and only if the group is an extraspecial $2$-group (the two exceptions are the dihedral group 
$D_8$ and the central product of two copies of $D_8$), see \cite[Theorem 6.2]{CT07}.

\section{The Plesken Lie algebra for associative algebras with an anti-involution}
\label{sec:Pleskeninvolution}

The aim of this section is to generalize the construction of the Plesken Lie algebra. Looking back at the 
construction of $\mathcal{L}(\mathbb{C}G)$ for a finite group $G$ in Section \ref{sec:Pleskengroup},
the crucial input was the associative algebra $\mathbb{C}G$ and the fact that each group element has 
an inverse, so that $\widehat{g}=g-g^{-1}$ and hence 
$\mathcal{L}(\mathbb{C}G) = \mathrm{span}\{\widehat{g}\,|\,g\in G\}$ could be defined. 

Let $A$ be an associative algebra over some field. This can be turned into a Lie algebra by setting
$[a,b]=ab-ba$ for all $a,b\in A$. 
Moreover, we assume that we have a linear map $\sigma:A\to A$
which is an anti-involution, that is, $\sigma(ab)=\sigma(b)\sigma(a)$ for all $a,b\in A$ and $\sigma^2$
is the identity map on $A$. 

Our motivating example is the group algebra $A=\mathbb{C}G$ with the anti-involution
$\sigma:\mathbb{C}G\to \mathbb{C}G$ given on basis elements by $\sigma(g)=g^{-1}$ and
extended linearly.

\begin{Definition} \label{def:L(A)}
Let $A$ be an associative algebra and $\sigma:A\to A$ an anti-involution. The {\em Plesken Lie algebra}
of $A$ is defined as
$$\mathcal{L}(A) = \mathrm{span}\{\widehat{a}:=a-\sigma(a)\,|\,a\in A\}.
$$
\end{Definition} 

For this definition to make sense, we have to show that the subspace $\mathcal{L}(A)$ of $A$ 
is indeed a Lie algebra (with the Lie bracket induced from $A$). 

\begin{Lemma}
Let $A$ be an associative algebra and $\sigma:A\to A$ an anti-involution.
\begin{enumerate}
\item[{(i)}] For all $a,b\in A$ we have 
$$[\widehat{a},\widehat{b}] = \widehat{ab} - \widehat{a\sigma(b)} - \widehat{\sigma(a)b}
+ \widehat{\sigma(a)\sigma(b)} \in \mathcal{L}(A).
$$
\item[{(ii)}] $\mathcal{L}(A)$ is a Lie algebra (with the Lie bracket induced from $A$).
\end{enumerate}
\end{Lemma}

\begin{proof}
For proving (i), let $a,b\in A$. Then we have (where for the fourth equation we crucially use that
$\sigma$ is an anti-involution):
\begin{eqnarray*}
[\widehat{a},\widehat{b}] & = & \widehat{a}\widehat{b} - \widehat{b}\widehat{a} 
 =  (a-\sigma(a))(b-\sigma(b)) - (b-\sigma(b))(a-\sigma(a)) \\
 & = & 
 ab - a\sigma(b)-\sigma(a)b+\sigma(a)\sigma(b) -ba+b\sigma(a)+\sigma(b)a-\sigma(b)\sigma(a) \\
 & = & 
 ab - \sigma(b)\sigma(a) - a\sigma(b) + b\sigma(a) - \sigma(a)b + \sigma(b)a + \sigma(a)\sigma(b)-ba \\
 & = & ab - \sigma(ab) - a\sigma(b) + \sigma(a\sigma(b)) - \sigma(a)b + \sigma(\sigma(a)b) + 
 \sigma(a)\sigma(b)-\sigma(\sigma(a)\sigma(b)) \\
 & = & \widehat{ab} - \widehat{a\sigma(b)} - \widehat{\sigma(a)b}+ \widehat{\sigma(a)\sigma(b)}
 \in \mathcal{L}(A).
\end{eqnarray*}
Statement (ii) follows from (i) since $\mathcal{L}(A)$ is a subspace by Definition \ref{def:L(A)}
and by (i) it is closed under the Lie bracket. Thus, $\mathcal{L}(A)$ is a Lie subalgebra of $A$
and hence itself a Lie algebra. 
\end{proof}

Note that for $A=\mathbb{C}G$ and the anti-involution $\sigma$ given by inverting the group elements, the 
Lie algebra $\mathcal{L}(A)$ is exactly the Plesken Lie algebra $\mathcal{L}(G)$ from 
\cite{CT07}, as discussed in Section \ref{sec:Pleskengroup}.

\section{Some small examples} \label{sec:examples}

In this section we discuss some well-known examples of associative algebras 
with an anti-involution and determine their Plesken Lie algebras. 


\subsection{Quaternions with conjugation} \label{subsec:quaternions}
We consider the quaternions
$$\mathbb{H} = \{a_0+a_1i+a_2j+a_3k\,|\,a_0,a_1,a_2,a_3\in \mathbb{R}\}
$$
which form a 4-dimensional associative algebra over $\mathbb{R}$. The multiplication in $\mathbb{H}$
is given by the usual rules, namely $i^2=j^2=k^2=-1$, $ij=k=-ji$, $ik=-j=-ki$, $jk=i=-kj$,
extended linearly. The conjugation map 
$$\sigma:\mathbb{H}\to \mathbb{H}~,~~a_0+a_1i+a_2j+a_3k\mapsto a_0-a_1i-a_2j-a_3k
$$
is an anti-involution (this is straightforward to check from the multiplication rules in $\mathbb{H}$,
e.g. $\sigma(ij)=\sigma(k)=-k=ji=(-j)(-i)=\sigma(j)\sigma(i)$). As a real vector space, the Plesken Lie
algebra of the quaternions is
\begin{eqnarray*}
\mathcal{L}(\mathbb{H}) & = & \mathrm{span}\{\widehat{a}=a-\sigma(a)\,|\,a\in \mathbb{H}\} \\
& = & \mathrm{span}\{\widehat{a}=2(a_1i+a_2j+a_3k)\,|\,a_1,a_2,a_3\in \mathbb{R}\} \\
& = & \mathrm{span} \{i,j,k\},
\end{eqnarray*}
a 3-dimensional Lie algebra over $\mathbb{R}$. The Lie brackets in $\mathcal{L}(\mathbb{H})$
are on the vector space basis given by
$$[i,j]=ij-ji=2k~,~~[i,k]=ik-ki=-2j~,~~[j,k]=jk-kj=2i.
$$
Scaling the basis elements to $e_1:=\frac{1}{2}i$, $e_2:=\frac{1}{2}j$, 
$e_3:=\frac{1}{2}k$, the Lie brackets take the form
$$[e_1,e_2] = e_3~,~~[e_1,e_3]=-e_2~,~~[e_2,e_3]=e_1.
$$
These are precisely the relations for the cross product (or vector product) of unit vectors in 
3-dimensional Euclidean space $\mathbb{R}^3$. We have shown: the Plesken Lie algebra
$\mathcal{L}(\mathbb{H})$ is isomorphic to the Lie algebra $\mathbb{R}^3$ with Lie bracket 
given by the cross product.

\subsection{Matrix algebras with transposition} \label{subsec:Mntrans}
For any $n\in \mathbb{N}$ we consider the algebra $A=M(n,\mathbb{C})$ of all $n\times n$-matrices
over $\mathbb{C}$. This is a $\mathbb{C}$-algebra of dimension $n^2$, with a basis given by
$\{E_{r,s}\,|\,1\le r,s\le n\}$ where $E_{r,s}$ denotes the matrix with a 1 in position $(r,s)$ and
0 otherwise. 

Transposition of matrices yields an anti-involution $\sigma:M(n,\mathbb{C})\to M(n,\mathbb{C})$.
Note that $\sigma(E_{r,s}) = E_{s,r}$ for all $1\le r,s\le n$. For the Plesken Lie algebra we get a vector space basis
$$\mathcal{L}(M_n(\mathbb{C})) = \mathrm{span} \{\widehat{E}_{r,s}:=E_{r,s}-E_{s,r}\,|\,1\le r<s\le n\}.
$$
In particular, $\mathcal{L}(M(n,\mathbb{C}))$ has dimension $\frac{n(n-1)}{2}$ as a vector space over
$\mathbb{C}$. More precisely, the Plesken Lie algebra $\mathcal{L}(M(n,\mathbb{C}))$ consists 
of the skew-symmetric $n\times n$-matrices over $\mathbb{C}$. 
The skew-symmetric $n\times n$-matrices over $\mathbb{C}$ form a Lie algebra, namely the 
orthogonal Lie algebra $\mathfrak{o}(n,\mathbb{C})$. 

Thus the Plesken Lie algebra of a matrix algebra with respect to transposition
is
$$\mathcal{L}(M(n,\mathbb{C})) = \mathfrak{o}(n,\mathbb{C})\mbox{~~for all $n\in \mathbb{N}$.}
$$

\subsection{Matrix algebras with conjugate transposition} \label{subsec:Mntransconj}
We consider again the associative $\mathbb{C}$-algebra $A=M(n,\mathbb{C})$ of all $n\times n$-matrices
over $\mathbb{C}$, but now with the anti-involution $\sigma:M(n,\mathbb{C})\to M(n,\mathbb{C})$,
$\sigma(M)=\overline{M}^T$ (where $\overline{M}$ is the matrix obtained from $M$ by complex
conjugation of each entry).  
We shall show that the Plesken Lie algebra
$$\mathcal{L}(M(n,\mathbb{C})) = \mathrm{span} \{M-\overline{M}^T\,|\,M\in M(n,\mathbb{C})\}
$$
contains all elementary matrices $E_{r,s}$, for $1\le r,s\le n$. For $r=s$ we have that
$$E_{r,r} = \frac{1}{2i} \left( i E_{r,r} - (-i)E_{r,r}\right) = 
\frac{1}{2i}\left( iE_{r,r} - \overline{iE_{r,r}}^T\right) \in  \mathcal{L}(M(n,\mathbb{C}))
$$
and for $r\neq s$ we have 
$$E_{r,s} = \frac{1}{2} \left( \left(E_{r,s}-\overline{E_{r,s}}^T\right) - i \left(i E_{r,s}-\overline{iE_{r,s}}^T\right)\right)
\in  \mathcal{L}(M(n,\mathbb{C})).$$
Hence, $\mathcal{L}(M(n,\mathbb{C})) = \mathrm{span}(\{E_{r,s}\,|\,1\le r,s\le n\}) = M(n,\mathbb{C})$.
Recall that the Lie bracket in $\mathcal{L}(M(n,\mathbb{C}))$ is given by the Lie bracket coming from the 
associative algebra $M(n,\mathbb{C})$. So 
for the Plesken Lie algebra of a matrix algebra with respect to conjugate transposition
we get the general linear Lie algebras
$$\mathcal{L}(M(n,\mathbb{C})) = \mathfrak{gl}(n,\mathbb{C})\mbox{~~for all $n\in \mathbb{N}$.}
$$


\subsection{Matrix algebras of algebras with an anti-involution} \label{subsec:matrixalginvol}
Let $A$ be an associative algebra with an anti-involution $\sigma:A\to A$. For any $n\in \mathbb{N}$,
the set of all $n\times n$-matrices matrices
with entries from $A$, 
$$M(n,A) = \{M=(m_{ij})\,|\,m_{ij}\in A\mbox{~for $1\le i,j\le n$}\}
$$
is an associative algebra. We define a map
$$\sigma_n:M(n,A)\to M(n,A)~,~~M=(m_{ij})\mapsto \sigma_n(M)=(\sigma(m_{ji})). 
$$
We claim that $\sigma_n$ is an anti-involution. 
In fact, let $M=(m_{ij})$ and $N=(n_{ij})$ 
be elements from $M(n,A)$. For any $i,j\in \{1,\ldots,n\}$ we compare the $(i,j)$-entries of 
$\sigma_n(MN)$ and $\sigma_n(N)\sigma_n(M)$. By definition of the map $\sigma_n$, the $(i,j)$-entry 
of $\sigma_n(MN)$ is 
\begin{equation} \label{eq:1}
\sigma (\sum_{k=1}^n m_{jk}n_{ki}) = \sum_{k=1}^n \sigma(n_{ki})\sigma(m_{jk})
\end{equation}
where the equation holds since $\sigma$ is an anti-homomorphism. 
On the other hand, the $(i,j)$-entry of $\sigma_n(N)\sigma_n(M)$ is
\begin{equation} \label{eq:2}
\sum_{k=1}^n \sigma(n_{ki})\sigma(m_{jk}).
\end{equation}
The equality of (\ref{eq:1}) and (\ref{eq:2}) shows that $\sigma_n$ is an anti-homomorphism. 
Finally, it is immediate from the definition that 
$\sigma_n$ is an involution (since $\sigma$ is an involution). 
\smallskip

The above observation implies that for every associative algebra $A$ with an anti-involution $\sigma$
and every $n\in \mathbb{N}$ one has the Plesken Lie algebra
$\mathcal{L}(M(n,A))$ with respect to the anti-involution $\sigma_n$.  

Note that the Plesken Lie algebras in Subsection \ref{subsec:Mntrans} (with $\sigma:\mathbb{C}\to \mathbb{C}$ 
the identity) and \ref{subsec:Mntransconj} (with $\sigma:\mathbb{C}\to \mathbb{C}$ the complex conjugation)
are special cases of this construction.

\section{Plesken Lie algebras for semisimple cellular algebras}
\label{sec:Pleskencellular}

Cellular algebras habe been introduced by Graham and Lehrer \cite{GL96} (subsequently, K\"onig and Xi
\cite{KX98} 
gave a more abstract definition and showed that the two definitions are equivalent). These are associative
algebras with a distinguished basis, a {\em cellular basis}, allowing a general description of the representations 
via {\em cell modules}. For later reference we include the precise definition of a cellular algebra here
(for simplicity we restrict to algebras over a field). 

\begin{Definition}{\cite[Definition 1.1]{GL96}} \label{def:cellular}
A {\em cellular algebra} is an associative algebra over a field $K$, together with a {\em cell datum}
$(\Lambda, M,C,i)$ where
\begin{enumerate}
\item[{(C1)}] $\Lambda$ is a finite partially ordered set and for each $\lambda\in \Lambda$ 
there is a non-empty finite set $M(\lambda)$ of indices such that 
$$C:\bigcup_{\lambda\in \Lambda} M(\lambda)\times M(\lambda)\to A
$$
is an injective map whose image is a $K$-basis of $A$. We write $C_{s,t}^{\lambda}:=C(s,t)$ for 
$s,t\in M(\lambda)$.
\item[{(C2)}] $i:A\to A$ is a $K$-linear anti-involution such that $i(C_{s,t}^{\lambda})= C_{t,s}^{\lambda}$
for all $\lambda\in \Lambda$ and $s,t\in M(\lambda)$. 
\item[{(C3)}] For all $\lambda\in \Lambda$, $s,t\in M(\lambda)$ and $a\in A$ there exist
$r_a(s',s)\in K$ such that
$$aC_{s,t}^{\lambda} \equiv \sum_{s'\in M(\lambda)} r_a(s',s) C_{s',t}^{\lambda}\,\mathrm{mod}\,A(<\lambda)
$$
where $A(<\lambda)$ is the $K$-subspace spanned by all basis elements $C_{v,w}^{\mu}$ with $\mu<\lambda$. 
\end{enumerate}
\end{Definition}

\begin{Example} \label{ex:cellular} Let $K$ be a field. 
\begin{enumerate}
\item[{(1)}] (Matrix algebras)
For every $n\in \mathbb{N}$ the $K$-algebra $M(n,K)$ of all $n\times n$-matrices over $K$ is a cellular algebra.
A cell datum is given by the partially ordered set $\Lambda=\{1\}$, the set $M(1)=\{1,\ldots,n\}$ of indices, the
injective map 
$$C:M(1)\times M(1)\to M(n,K)~,~~C(s,t)=E_{s,t}
$$
and the anti-involution
$i:M(n,K)\to M(n,K)$ given by transposition of matrices. Note that the scalars appearing in Definition \ref{def:cellular}\,(C3)
are in this example given by $r_a(s',s)=a_{s',s}$, the entry of $a\in M(n,K)$ at position $(s',s)$ (note that 
here we have $A(<\lambda)=0$ since $\Lambda$ has only one element and the congruence in (C3) becomes an equality). 
\item[{(2)}] (Symmetric group algebras) For any $n\in \mathbb{N}$ the group algebra of the
symmetric group $S_n$ over $K$ is a cellular algebra, see \cite[Example 1.2]{GL96}. A cell datum $(\Lambda, M,C,i)$
is given by the set $\Lambda$ of partitions of $n$ with respect to the dominant order, the sets $M(\lambda)$ of standard
Young tableaux of shape $\lambda$, the basis elements $C_{s,t}^{\lambda} = w\in S_n$ where $\lambda, s,t$ are
uniquely determined by $w$ (via the Robinson-Schenstedt correspondence) 
and the anti-involution $i:KS_n\to KS_n$ given by $i(w)=w^{-1}$. For more details we refer to \cite{GL96}.  
\end{enumerate}
\end{Example}

\begin{Lemma} \label{lem:eigenspace}
Let $A$ be a cellular algebra over a field $K$, with cell datum $(\Lambda,M,C,i)$. Then the Plesken Lie algebra
(cf. Definition \ref{def:L(A)})
$$\mathcal{L}(A) = \mathrm{span}\{ \widehat{a}=a-i(a)\,|\,a\in A\}
$$
is equal to the $(-1)$-eigenspace of the anti-involution $i$, that is,
$$\mathcal{L}(A) = \{a\in A\,|\, i(a)=-a\}.
$$
\end{Lemma}

\begin{proof}
We denote the $(-1)$-eigenspace by $E=\{a\in A\,|\, i(a)=-a\}$. We first observe that $\mathcal{L}(A)\subseteq E$.
In fact, for every generating element $\widehat{a}=a-i(a)$ of $\mathcal{L}(A)$ we have
\begin{equation} \label{eq:widehata}
i(\widehat{a}) = i(a-i(a)) = i(a) - i^2(a) = i(a)-a = - \widehat{a}.
\end{equation}
Conversely, let $a\in E$. We write $a$ as a linear combination in the cellular basis from Definition \ref{def:cellular}\,(C1),
\begin{equation} \label{eq:alincomb}
a = \sum_{\lambda\in \Lambda} \sum_{s,t\in M(\lambda)} \alpha_{s,t}^{\lambda} C_{s,t}^{\lambda}.
\end{equation}
Using the assumption $a\in E$ and Definition \ref{def:cellular}\,(C2) we obtain
$$
\sum_{\lambda\in \Lambda} \sum_{s,t\in M(\lambda)} \alpha_{s,t}^{\lambda} C_{s,t}^{\lambda}
  =  a = -i(a) =
 \sum_{\lambda\in \Lambda} \sum_{s,t\in M(\lambda)} (-\alpha_{s,t}^{\lambda}) C_{t,s}^{\lambda}.
$$
Since $\{C_{s,t}^{\lambda}\,|\,\lambda\in \Lambda,s,t\in M(\lambda)\}$ is a $K$-basis of $A$
(by Definition \ref{def:cellular}\,(C1)), we can compare
coefficients and get
$$\alpha_{s,t}^{\lambda} = - \alpha_{t,s}^{\lambda}\mbox{~~for all $\lambda\in \Lambda$ and $s,t\in M(\lambda)$}.
$$
This implies that in equation (\ref{eq:alincomb}), $a$ is actually a linear combination of elements of the form
$$C_{s,t}^{\lambda}-C_{t,s}^{\lambda} = C_{s,t}^{\lambda}-i(C_{s,t}^{\lambda})
= \widehat{C_{s,t}^{\lambda}}
$$
and hence
$E\subseteq \mathcal{L}(A)$.
\end{proof}

\bigskip

Let $A$ be a cellular algebra over a field $K$ with cell datum $(\Lambda,M,C,i)$ as in Definition \ref{def:cellular}.
We recall from \cite{GL96} the definition of a set of modules which play a crucial role for 
the representation theory of $A$. 

\begin{Definition}{\cite[Definition 2.1]{GL96}} \label{def:cellmodule}
For each $\lambda\in\Lambda$ the {\em cell module} of $A$ corresponding to $\lambda$ is 
the $K$-vector space with basis $\{C_s\,|\,s\in M(\lambda)\}$ and $A$-action given by
$$aC_s =  \sum_{s'\in M(\lambda)} r_a(s',s) C_{s'}.
$$
\end{Definition}

\begin{Example}
For the cellular algebra $M(n,K)$ (cf. Example \ref{ex:cellular}\,(1)) the only cell module 
is $W(1)$ with basis $\{C_1,\ldots,C_n\}$ and $M(n,K)$-action given by
$$aC_s = \sum_{s'=1}^n a_{s',s} C_{s'}.
$$
Note that this module is isomorphic to the natural $M(n,K)$-module $K^n$ of column vectors via the
isomorphism mapping $C_s$ to the unit vector $e_s$, for $s=1,\ldots,n$.  
\end{Example}

\medskip

The following bilinear forms 
are important for 
the structure of the cell modules $W(\lambda)$.

\begin{Definition}{\cite[Lemma 1.7, Definition 2.3]{GL96}} \label{def:bilinearform}
Let $\lambda\in \Lambda$.
\begin{enumerate}
\item[{(i)}] For every $s,t,u,v\in M(\lambda)$ there are scalars $\phi_{\lambda}(t,u)\in K$
(independent of $s$ and $v$) such that 
$$C_{s,t}^{\lambda}C_{u,v}^{\lambda} \equiv \phi_{\lambda}(t,u) C_{s,v}^{\lambda}\,\mathrm{mod}\,A(<\lambda).
$$
\item[{(ii)}] The bilinear form
$\phi_{\lambda}:W(\lambda)\times W(\lambda)\to K$ is defined by the values $\phi_{\lambda}(t,u)$ for any $t,u\in M(\lambda)$,
and extended bilinearly.
\end{enumerate}
\end{Definition}

The following result collects some useful properties of the bilinear forms $\phi_{\lambda}$. 

\begin{Proposition}{\cite[Proposition 2.4]{GL96}} \label{def:prop_bf}
With the above notation, the following holds for the bilinear form $\phi_{\lambda}$ for each $\lambda\in \Lambda$.
\begin{enumerate}
\item[{(i)}] $\phi_{\lambda}$ is symmetric, that is, $\phi_{\lambda}(x,y)=\phi_{\lambda}(y,x)$
for all $x,y\in W(\lambda)$. 
\item[{(ii)}] With the anti-involution $i:A\to A$ we have
$$\phi_{\lambda}(i(a)x,y) = \phi_{\lambda}(x,ay)
$$
for all $a\in A$ and $x,y\in W(\lambda)$. 
\end{enumerate}
\end{Proposition}

\begin{Example}
We consider again the cellular algebra $M(n,K)$. As observed in Example \ref{ex:cellular}\,(1)
we have $C_{s,t}^1=E_{s,t}$ and hence the multiplication formula for matrix units
applied to Definition \ref{def:bilinearform}\,(i) yields the bilinear form
$$\phi_1(t,u) = \left\{ \begin{array}{ll} 1 & \mbox{if $t=u$} \\ 0 & \mbox{if $t\neq u$}
\end{array} \right. 
$$
extended bilinearly. Clearly, this bilinear form is symmetric and non-degenerate. 
\end{Example}

\medskip

The following result characterizes which cellular algebras are semisimple.

\begin{Theorem}{\cite[Theorem 3.8]{GL96}} \label{thm:semisimplecellular}
Let $A$ be a cellular algebra with cell datum $(\Lambda,M,C,i)$. Then the following are equivalent.
\begin{enumerate}
\item[{(i)}] $A$ is a semisimple algebra.
\item[{(ii)}] The non-zero cell modules $W(\lambda)$ are simple and pairwise non-isomorphic. Moreover, they form
a complete set of simple $A$-modules, up to isomorphism.  
\item[{(iii)}] For each $\lambda\in \Lambda$ the symmetric bilinear form $\phi_{\lambda}$ is non-degenerate. 
\end{enumerate}
\end{Theorem}

\begin{Definition} \label{def:phi}
Let $A$ be a semisimple cellular algebra with cell datum $(\Lambda,M,C,i)$. By Theorem \ref{thm:semisimplecellular},
there is a non-degenerate symmetric bilinear form $\phi_{\lambda}$ on each non-zero simple module $W(\lambda)$.
Set $\Lambda'=\{\lambda\in \Lambda\,|\,W(\lambda)\neq 0\}$ and consider the $A$-module
$$W:= \bigoplus_{\lambda\in \Lambda'} W(\lambda)
$$ 
(with diagonal $A$-action).
We define a bilinear form $\phi:W\times W\to K$ as the sum of the bilinear forms $\phi_{\lambda}$, that is
for $x=\sum_{\lambda\in \Lambda'} x_{\lambda}$  and $y=\sum_{\lambda\in \Lambda'} y_{\lambda}$ in $W$
we set
$$\phi(x,y)= \sum_{\lambda\in \Lambda'} \phi_{\lambda}(x_{\lambda},y_{\lambda}).
$$
This is a non-degenerate symmetric bilinear form and we have
\begin{equation} \label{eq:phi_ia}
\phi(i(a)x,y) = \phi(x,ay)
\end{equation}
for all $a\in A$ and $x,y\in W$ 
(as each $\phi_{\lambda}$ has these properties).  
\end{Definition}

\begin{Lemma} \label{lem:LA}
Let $A$ be a semisimple cellular algebra with cell datum $(\Lambda,M,C,i)$, and let $\phi:W\times W\to K$ be the 
non-degenerate symmetric bilinear form from Definition \ref{def:phi}
on the $A$-module $W=\bigoplus_{\lambda\in \Lambda'} W(\lambda)$,
the direct sum of a complete set of pairwise non-isomorphic simple $A$-modules. Then for the Plesken Lie algebra of $A$
we have
$$\mathcal{L}(A) = \{a\in A\,|\,\phi(ax,y)+\phi(x,ay)=0\mbox{~for all $x,y\in W$}\}.
$$ 
\end{Lemma}

\begin{proof}
We set $U:=\{a\in A\,|\,\phi(ax,y)+\phi(x,ay)=0\mbox{~for all $x,y\in W$}\}$
for abbreviation. 

We know from Lemma \ref{lem:eigenspace} that 
\begin{equation} \label{eq:LAi}
\mathcal{L}(A)=\{a\in A\,|\,i(a)=-a\}.
\end{equation}
Suppose first that $a\in A$ satisfies $i(a)=-a$. Then Equation (\ref{eq:phi_ia}) yields that for all $x,y\in W$ we have
$$\phi(-ax,y)=\phi(x,ay)$$
and hence $a\in U$.

Conversely, suppose that $a\in U$. Then equation (\ref{eq:phi_ia}) gives
$$\phi(ax+i(a)x,y)=0\mbox{~~for all $x,y\in W$}.
$$
Since $\phi$ is non-degenerate this implies that
$$(a+i(a))x=0\mbox{~~for all $x\in W$}.
$$
Since $W=\bigoplus_{\lambda\in \Lambda'} W(\lambda)$ this in particular means that for all $\lambda\in \Lambda'$
we get
$$(a+i(a))x_{\lambda}=0\mbox{~~for all $x_{\lambda}\in W(\lambda)$}.
$$
As a consequence, $a+i(a)\in A$ annihilates all simple $A$-modules (cf. Theorem \ref{thm:semisimplecellular}\,(ii)).
On the other hand, the intersection of the annihilators of all simple $A$-modules is the Jacobson radical of
$A$ which is 0 since $A$ is semisimple (see e.g. \cite[Theorem 4.23]{EH18}). Thus, $a+i(a)=0$ and then
$a\in \mathcal{L}(A)$ follows from Equation (\ref{eq:LAi}).
\end{proof}

\medskip

We can now prove the main result of this section, which describes the structure of the Plesken Lie algebras for
semisimple cellular algebras. After the above preparations, the proof becomes similar to the proof of
\cite[Theorem 5.1]{CT07} for group algebras.

\begin{Theorem} \label{thm:Plesken_cellular} 
Let $A$ be a semisimple cellular algebra over $\mathbb{C}$ with cell datum $(\Lambda,M,C,i)$
and let $W(\lambda)$ for 
$\lambda\in \Lambda'$ be the simple $A$-modules (up to isomorphism). 
Then the Plesken Lie algebra of $A$ with respect to the anti-involution $i$ is the 
direct sum of orthogonal Lie algebras, more precisely:
$$\mathcal{L}(A) \cong \bigoplus_{\lambda\in \Lambda'} \,\,\mathfrak{o}(\mathrm{dim}\,W(\lambda),\mathbb{C}).
$$
\end{Theorem}

\begin{proof}
Recall that we denoted $\Lambda'=\{\lambda\in \Lambda\,|\,W(\lambda)\neq 0\}$ 
and that $\{W(\lambda)\,|\,\lambda\in \Lambda'\}$ is a complete set of simple $A$-modules,
up to isomorphism.  
As $A$ is a semisimple algebra, it decomposes as an $A$-module into a direct sum of simple modules,
$$A \cong \bigoplus_{\lambda\in \Lambda'} W(\lambda)^{d_{\lambda}}.
$$
 By Schur's lemma and Artin-Wedderburn theory 
we have that
\begin{equation} \label{eq:AW}
A\cong \mathrm{End}_A(A)^{op} 
\cong
\bigoplus_{\lambda\in \Lambda'} M(d_{\lambda},\mathrm{End}_A(W(\lambda))^{op}
\end{equation}
(see \cite[Lemma 5.4]{EH18} for the first isomorphism, \cite[Lemma 5.6,Theorem 5.7]{EH18} for the second isomorphism
and use that $\mathrm{Hom}_A(W(\lambda),W(\mu))=0$ for $\lambda\neq \mu$ by Schur's lemma). Moreover, 
the multiplicities are given by $d_{\lambda}= \mathrm{dim}\,W(\lambda)$
(using that $\mathbb{C}$ is algebraically closed, see e.g. \cite[Corollary 5.11\,(b)]{EH18}).
Moreover, the endomorphism algebras $\mathrm{End}_A(W(\lambda))$ are division algebras (by Schur's lemma)
and since $\mathbb{C}$ is algebraically closed, we even have that $\mathrm{End}_A(W(\lambda))\cong \mathbb{C}$
are one-dimensional. Thus we can continue the above sequence of isomorphisms to get
\begin{equation} \label{eq:AW2}
A\cong \bigoplus_{\lambda\in \Lambda'} M(d_{\lambda},\mathrm{End}_A(W(\lambda))^{op}
\cong \bigoplus_{\lambda\in \Lambda'} M(d_{\lambda},\mathbb{C})^{op}
\cong \bigoplus_{\lambda\in \Lambda'} M(d_{\lambda},\mathbb{C})
\end{equation}
(where the latter isomorphism is given by transposition of matrices). 

The crucial observation now is that the direct sum decomposition in (\ref{eq:AW}) is compatible 
with respect to the Lie bracket on $A$. In fact, the first isomorphism is mapping $a\in A$ to 
$\varphi_a\in \mathrm{End}_A(A)^{op}$ with $\varphi_a(1)=a$ (right multiplication with $a$). 
Then a Lie bracket $[a,b]$ in $A$ is mapped to 
$\varphi_{[a,b]}=[\varphi_a,\varphi_b]$ since
\begin{eqnarray*}
[\varphi_a,\varphi_b](1) & = & (\varphi_a\ast\varphi_b-\varphi_b\ast\varphi_a)(1)
= (\varphi_b\varphi_a-\varphi_a\varphi_b)(1) \\
& = & \varphi_b(a)-\varphi_a(b) = a\varphi_b(1)-b\varphi_a(1) = ab-ba = [a,b]
\end{eqnarray*}
where $\ast$ denotes the opposite multiplication. 
The second isomorphism in (\ref{eq:AW}) is obtained by considering the endomorphisms componentwise,
hence it follows that the Lie brackets are preserved (which are on both sides given by commutators).  

We have shown in Lemma \ref{lem:LA} that 
\begin{equation} \label{eq:LAcond}
\mathcal{L}(A) = \{a\in A\,|\,\phi(ax,y)+\phi(x,ay)=0\mbox{~for all $x,y\in W$}\}\subseteq A.
\end{equation}
To complete the proof it suffices to see that the condition $\phi(ax,y)+\phi(x,ay)=0$ characterizing 
the elements of the Plesken Lie algebra $\mathcal{L}(A)\subseteq A$
transfers under the sequence of isomorphisms in (\ref{eq:AW}) and (\ref{eq:AW2})
to the orthogonal matrices in $\bigoplus_{\lambda\in \Lambda'} M(d_{\lambda},\mathbb{C})$. 

The non-degenerate symmetric bilinear form $\phi$ on $W$ is defined as the sum of the non-degenerate 
symmetric bilinear forms $\phi_{\lambda}$ on $W(\lambda)$. Hence the condition in (\ref{eq:LAcond}) 
for $\phi$ holds for all $x,y\in W$ if and only if the same condition holds for any of the $\phi_{\lambda}$ 
for all $x,y\in W(\lambda)$. 

We can choose bases in $W_{\lambda}$, and then the condition in (\ref{eq:LAcond}) can be expressed with matrices
and coordinate vectors, where for simplicity we use the same notation. 
Moreover, since $\phi_{\lambda}$ is non-degenerate, we can choose a basis of $W(\lambda)$ such that
the Gram matrix of $\phi_{\lambda}$ is the identity matrix. Then 
$$\phi_{\lambda}(ax,y) + \phi_{\lambda}(x,ay) = (ax)^T y+ x(ay)^T = x^T a^T y + x^T (ay) = x^T (a^t+a)y.
$$
This is zero for all vectors $x,y$ if and only if $a^t+a=0$, that is, if and only if $a$ is contained in the 
orthogonal Lie algebra $\mathfrak{o}(\mathrm{dim}\,W(\lambda),\mathbb{C})= \mathfrak{o}(d_{\lambda},\mathbb{C})
\subseteq M(d_{\lambda},\mathbb{C})$. 
\end{proof}

\section{The Plesken Lie algebra for diagram algebras}
\label{sec:diagram}

In this section we illustrate our main result Theorem \ref{thm:Plesken_cellular} by considering 
planar rook algebras and Temperley-Lieb algebras. These algebras are combinatorially defined in terms of
diagrams and they appear as subalgebras of partition algebras. 
The partition algebras are associative algebras with bases consisting of certain diagrams and multiplication 
given by concatenation of these diagrams. They have been introduced some 30 years ago independently by 
V.\,Jones \cite{Jo94} and by
P.\,Martin \cite{Ma94}. The partition algebras contain as subalgebras numerous 
other prominent algebras, like the Temperley-Lieb algebras, the Brauer algebras, the Motzkin algebras,
the group algebras of symmetric groups and many other diagram algebras. 

For all diagram algebras, there is a natural anti-involution given by
reflecting the diagrams (interchanging top and bottom nodes). This particular anti-involution is a crucial 
ingredient of the combinatorial and representation-theoretic structure of the partition algebras (and their subalgebras),
namely they are cellular algebras \cite[Theorem 4.1]{Xi99}. 

For this anti-involution, we then have a corresponding Plesken Lie algebra for all the diagram algebras
(cf. Definition \ref{def:L(A)}). The aim 
of this section is to 
determine the structure of the Plesken Lie algebra for two particular classes of diagram algebras, namely the
planar rook algebras and the Temperley-Lieb algebras (with some restriction on the parameter). 
\medskip

We start by considering the planar rook algebras. 

\begin{Definition}
For $n\in \mathbb{N}$ we consider diagrams with $n$ nodes in the top row and $n$ nodes in the bottom row. 
An {\em arc} in the diagram is connecting some node in the top row with some node in the bottom row. 
A {\em planar rook diagram} is such a diagram in which the arcs can be drawn without any crossings, and this also 
excludes common endpoints of arcs. For examples of planar rook diagrams see Figure \ref{fig:planarrookn3}.

The {\em planar rook algebra} $PR(n)$ over a field $K$ is the $K$-vector space with basis consisting of all planar 
rook diagrams with $n$ nodes in the top and bottom rows. The multiplication in $PR(n)$ is given by concatenation
of diagrams, extended linearly. For an example of a concatenation of diagrams see Figure \ref{fig:concatplanarrook}. 
\end{Definition}

\begin{figure}
\begin{center}
\begin{tikzpicture}
                 \fill[black!15!white] (0,0) rectangle (2,1);
     \foreach \n in {0,1,...,2}
        \foreach \m in {0,1}
            \node at (\n,\m)[circle,fill,inner sep=1.5pt]{};
     \draw (0,1) .. controls (0,0.5) and (0,0.5) .. (0,0);
     \draw (1,1) .. controls (1,0.5) and (1,0.5) .. (1,0);
     \draw (2,1) .. controls (2,0.5) and (2,0.5) .. (2,0);
     \fill[black!15!white] (3,0) rectangle (5,1);
     \foreach \n in {3,4,...,5}
        \foreach \m in {0,1}
            \node at (\n,\m)[circle,fill,inner sep=1.5pt]{};
     \draw (3,1) .. controls (3,0.5) and (4,0.5) .. (4,0);
     \draw (5,1) .. controls (5,0.5) and (5,0.5) .. (5,0);
     \fill[black!15!white] (6,0) rectangle (8,1);
     \foreach \n in {6,7,...,8}
        \foreach \m in {0,1}
            \node at (\n,\m)[circle,fill,inner sep=1.5pt]{};
     \draw (6,1) .. controls (6,0.5) and (6,0.5) .. (6,0);
     \draw (7,1) .. controls (7,0.5) and (8,0.5) .. (8,0);
     \fill[black!15!white] (9,0) rectangle (11,1);
     \foreach \n in {9,10,...,11}
        \foreach \m in {0,1}
            \node at (\n,\m)[circle,fill,inner sep=1.5pt]{};
     \draw (10,1) .. controls (10,0.5) and (9,0.5) .. (9,0);
     \draw (11,1) .. controls (11,0.5) and (11,0.5) .. (11,0);
     \fill[black!15!white] (12,0) rectangle (14,1);
     \foreach \n in {12,13,...,14}
        \foreach \m in {0,1}
            \node at (\n,\m)[circle,fill,inner sep=1.5pt]{};
     \draw (12,1) .. controls (12,0.5) and (12,0.5) .. (12,0);
     \draw (14,1) .. controls (14,0.5) and (13,0.5) .. (13,0);
     \fill[black!15!white] (0,-2) rectangle (2,-1);
     \foreach \n in {0,1,...,2}
        \foreach \m in {-2,-1}
            \node at (\n,\m)[circle,fill,inner sep=1.5pt]{};
     \draw (0,-1) .. controls (0,-1.5) and (0,-1.5) .. (0,-2);
     \draw (1,-1) .. controls (1,-1.5) and (1,-1.5) .. (1,-2);
     \fill[black!15!white] (3,-2) rectangle (5,-1);
     \foreach \n in {3,4,...,5}
        \foreach \m in {-2,-1}
            \node at (\n,\m)[circle,fill,inner sep=1.5pt]{};
     \draw (3,-1) .. controls (3,-1.5) and (3,-1.5) .. (3,-2);
     \draw (5,-1) .. controls (5,-1.5) and (5,-1.5) .. (5,-2);
     \fill[black!15!white] (6,-2) rectangle (8,-1);
     \foreach \n in {6,7,...,8}
        \foreach \m in {-2,-1}
            \node at (\n,\m)[circle,fill,inner sep=1.5pt]{};
     \draw (7,-1) .. controls (7,-1.5) and (7,-1.5) .. (7,-2);
     \draw (8,-1) .. controls (8,-1.5) and (8,-1.5) .. (8,-2);
     \fill[black!15!white] (9,-2) rectangle (11,-1);
     \foreach \n in {9,10,...,11}
        \foreach \m in {-2,-1}
            \node at (\n,\m)[circle,fill,inner sep=1.5pt]{};
     \draw (9,-1) .. controls (9,-1.5) and (10,-1.5) .. (10,-2);
     \draw (10,-1) .. controls (10,-1.5) and (11,-1.5) .. (11,-2);
     \fill[black!15!white] (12,-2) rectangle (14,-1);
     \foreach \n in {12,13,...,14}
        \foreach \m in {-2,-1}
            \node at (\n,\m)[circle,fill,inner sep=1.5pt]{};
     \draw (13,-1) .. controls (13,-1.5) and (12,-1.5) .. (12,-2);
     \draw (14,-1) .. controls (14,-1.5) and (13,-1.5) .. (13,-2);
     \fill[black!15!white] (0,-4) rectangle (2,-3);
     \foreach \n in {0,1,...,2}
        \foreach \m in {-4,-3}
            \node at (\n,\m)[circle,fill,inner sep=1.5pt]{};
     \draw (0,-3) .. controls (0,-3.5) and (0,-3.5) .. (0,-4);
     \fill[black!15!white] (3,-4) rectangle (5,-3);
     \foreach \n in {3,4,...,5}
        \foreach \m in {-4,-3}
            \node at (\n,\m)[circle,fill,inner sep=1.5pt]{};
     \draw (4,-3) .. controls (4,-3.5) and (4,-3.5) .. (4,-4);
     \fill[black!15!white] (6,-4) rectangle (8,-3);
     \foreach \n in {6,7,...,8}
        \foreach \m in {-4,-3}
            \node at (\n,\m)[circle,fill,inner sep=1.5pt]{};
     \draw (8,-3) .. controls (8,-3.5) and (8,-3.5) .. (8,-4);
     \fill[black!15!white] (9,-4) rectangle (11,-3);
     \foreach \n in {9,10,...,11}
        \foreach \m in {-4,-3}
            \node at (\n,\m)[circle,fill,inner sep=1.5pt]{};
     \draw (9,-3) .. controls (9,-3.5) and (10,-3.5) .. (10,-4);
     \fill[black!15!white] (12,-4) rectangle (14,-3);
     \foreach \n in {12,13,...,14}
        \foreach \m in {-4,-3}
            \node at (\n,\m)[circle,fill,inner sep=1.5pt]{};
     \draw (13,-3) .. controls (13,-3.5) and (14,-3.5) .. (14,-4);
     \fill[black!15!white] (0,-6) rectangle (2,-5);
     \foreach \n in {0,1,...,2}
        \foreach \m in {-6,-5}
            \node at (\n,\m)[circle,fill,inner sep=1.5pt]{};
     \draw (1,-5) .. controls (1,-5.5) and (0,-5.5) .. (0,-6);
     \fill[black!15!white] (3,-6) rectangle (5,-5);
     \foreach \n in {3,4,...,5}
        \foreach \m in {-6,-5}
            \node at (\n,\m)[circle,fill,inner sep=1.5pt]{};
     \draw (5,-5) .. controls (5,-5.5) and (4,-5.5) .. (4,-6);
     \fill[black!15!white] (6,-6) rectangle (8,-5);
     \foreach \n in {6,7,...,8}
        \foreach \m in {-6,-5}
            \node at (\n,\m)[circle,fill,inner sep=1.5pt]{};
     \draw (6,-5) .. controls (6,-5.5) and (8,-5.5) .. (8,-6);
     \fill[black!15!white] (9,-6) rectangle (11,-5);
     \foreach \n in {9,10,...,11}
        \foreach \m in {-6,-5}
            \node at (\n,\m)[circle,fill,inner sep=1.5pt]{};
     \draw (11,-5) .. controls (11,-5.5) and (9,-5.5) .. (9,-6);
     \fill[black!15!white] (12,-6) rectangle (14,-5);
     \foreach \n in {12,13,...,14}
        \foreach \m in {-6,-5}
            \node at (\n,\m)[circle,fill,inner sep=1.5pt]{};
        \end{tikzpicture}
        \caption{All planar rook diagrams for $n=3$.}
    \label{fig:planarrookn3}
        \end{center}
        \end{figure}
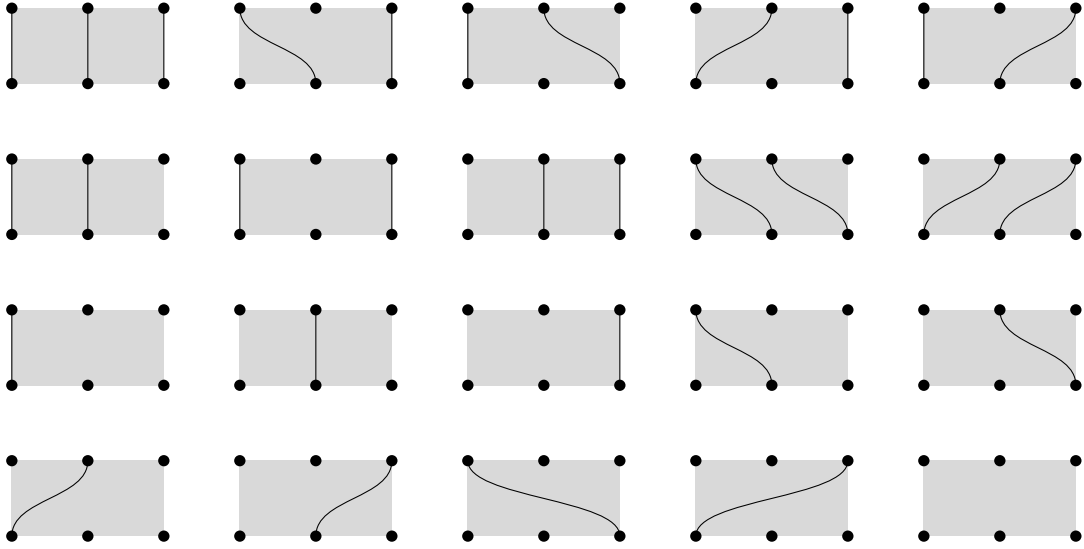

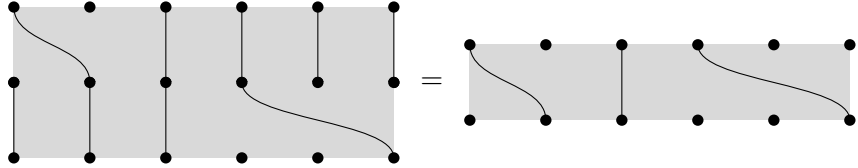
\begin{figure}
\begin{center}
    \begin{tikzpicture}
             \fill[black!15!white] (0,0) rectangle (5,1);
     \foreach \n in {0,1,...,5}
        \foreach \m in {0,1}
            \node at (\n,\m)[circle,fill,inner sep=1.5pt]{};
     \draw (0,1) .. controls (0,0.5) and (1,0.5) .. (1,0);
     \draw (2,1) .. controls (2,0.5) and (2,0.5) .. (2,0);
     \draw (3,1) .. controls (3,0.5) and (3,0.5) .. (3,0);
     \draw (4,1) .. controls (4,0.5) and (4,0.5) .. (4,0);
     \draw (5,1) .. controls (5,0.5) and (5,0.5) .. (5,0);
     \fill[black!15!white] (0,-1) rectangle (5,0);
     \foreach \n in {0,1,...,5}
        \foreach \m in {-1,0}
            \node at (\n,\m)[circle,fill,inner sep=1.5pt]{};
     \draw (0,0) .. controls (0,-0.5) and (0,-0.5) .. (0,-1);
     \draw (1,0) .. controls (1,-0.5) and (1,-0.5) .. (1,-1);
     \draw (2,0) .. controls (2,-0.5) and (2,-0.5) .. (2,-1);
     \draw (3,0) .. controls (3,-0.5) and (5,-0.5) .. (5,-1);
          \fill[black!15!white] (6,-0.5) rectangle (11,0.5);
     \foreach \n in {6,7,...,11}
        \foreach \m in {-0.5,0.5}
            \node at (\n,\m)[circle,fill,inner sep=1.5pt]{};
     \draw (6,0.5) .. controls (6,0) and (7,0) .. (7,-0.5);
     \draw (8,0.5) .. controls (8,0) and (8,0) .. (8,-0.5);
     \draw (9,0.5) .. controls (9,0) and (11,0) .. (11,-0.5);
     \node at (5.5,0){$=$};
    \end{tikzpicture}
    \caption{A concatenation of planar rook diagrams.}
    \label{fig:concatplanarrook}
\end{center}
\end{figure}


\begin{Remark} \label{rem:planarrook}
For every $n\in \mathbb{N}$ the planar rook algebra $PR(n)$ over any field $K$ is a
cellular algebra . 
\medskip

We have to show the properties (C1), (C2), (C3) from Definition \ref{def:cellular}.
\begin{enumerate}
\item[{(C1)}] Let $\Lambda = \{0,1,\ldots,n\}$ with the usual ordering 
(these numbers will refer to the number of arcs in a diagram). Since no crossings
of arcs are allowed, each planar rook diagram is uniquely determined by the end points of the arcs in the 
top and bottom rows. Thus for $\lambda\in \Lambda=\{0,1,\ldots,n\}$ we set $M(\lambda)$ as the set of
$\lambda$-element subsets of  $\Lambda=\{0,1,\ldots,n\}$. For any $\lambda$-element subset 
$s,t\in M(\lambda)$ we then take $C_{s,t}^{\lambda}$ as the diagram with $\lambda$ arcs having end points 
from $s$ in the top row and from $t$ in the bottom row. 
\item[{(C2)}] The anti-involution $i:A\to A$ is given on the basis of $PR(n)$ by swapping the
top and bottom row of the diagrams.
In particular, we have $i(C_{s,t}^{\lambda})=C_{t,s}^{\lambda}$.  
\item[{(C3)}] The multiplication in $PR(n)$ is given by concatenation of diagrams, extended linearly. 
When concatenating any diagram $C_{u,v}^{\mu}$ with $C_{s,t}^{\lambda}$ the number of arcs of the product 
$C_{u,v}^{\mu}C_{s,t}^{\lambda}$ is at most $\lambda$, and it is equal to $\lambda$ if and only if $v\supseteq s$. 
Moreover, the end points of the concatenated diagram in the bottom row will be a subset of $t$. By linear extension, this implies that
for each $a\in PR(n)$ (a linear combination of diagrams) there exist $r_a(s',s)\in K$ (independent of $t$) such 
that 
$$aC_{s,t}^{\lambda} \equiv \sum_{s'\in M(\lambda)} r_a(s',s) C_{s',t}^{\lambda}\,\mathrm{mod}\,A(<\lambda)
$$
where $A(<\lambda)$ is the subspace spanned by diagrams with fewer than $\lambda$ arcs. 
\end{enumerate}
\end{Remark}

The first application  of our Theorem \ref{thm:Plesken_cellular}
recovers a result from the second author's master thesis \cite{W25}, which there has been
obtained in a more direct way, without using the machinery of cellular algebras.

\begin{Corollary} \label{cor:planarrook}
Let $PR(n)$ be the planar rook algebra over $\mathbb{C}$ (for $n\in \mathbb{N}$). Then for the 
Plesken Lie algebra (with respect to the anti-involution given by reflecting diagrams) we have
$$\mathcal{L}(PR(n)) \cong \bigoplus_{k=0}^{n} \,\,\mathfrak{o}\left({n\choose k},\mathbb{C}\right).
$$ 
\end{Corollary}

\begin{proof}
We have seen in Remark \ref{rem:planarrook} that for every $n\in \mathbb{N}$ the algebra $PR(n)$ is cellular. 
Moreover, over $\mathbb{C}$, the algebra 
$PR(n)$ is semisimple,
see \cite[Theorem 3.2]{FHH09}. The regular module decomposes as
$$PR(n) = \bigoplus_{k=0}^n \bigoplus_{|T|=k} W_T^{n,k}
$$
where $W_T^{n,k}$ is spanned by the diagrams with exactly $k$ arcs and $T$ being the set of lower vertices 
attached to these $k$ arcs. For fixed $k$, the modules 
for different sets $T$ are isomorphic, and we denote these simple modules by $W(k)$ for $k=0,\ldots,n$. Then
$\{W(k)\,|\,0\le k\le n\}$ is a complete set of simple $PR(n)$-modules, up to isomorphism, see
\cite[Theorem 3.2\,(b)]{FHH09}, and the dimension of $W(k)$ is ${n\choose k}$ (as for a fixed set $T$ there 
are ${n\choose k}$ choices for the upper vertices attached to the $k$ arcs). Then Theorem \ref{thm:Plesken_cellular} 
implies the claim of the corollary.
\end{proof}

\medskip

Next we consider Temperley-Lieb algebras. The Temperley-Lieb algebras have strong links to mathematical physics,
where they appear in connection with the Potts models, ice-type models, or Andrews-Baxter-Forrester models.
The relevant diagrams are similar to the ones for the planar rook algebras,
but now we also allow arcs connecting nodes in the same row. 

\begin{Definition}
For $n\in \mathbb{N}$ we consider diagrams with $n$ nodes in the top row and $n$ nodes in the bottom row. 
An {\em arc} in the diagram is connecting two different nodes. 
A {\em Temperley-Lieb diagram} is such a diagram in which the arcs can be drawn without any crossings, and this also 
excludes common endpoints of arcs. For examples of Temperley-Lieb diagrams see Figure \ref{fig:temperleyliebn4}.

The {\em Temperley-Lieb algebra} $TL_{\delta}(n)$ over a field $K$ for the parameter $\delta\in K$
is the $K$-vector space with basis consisting of all Temperley-Lieb diagrams with $n$ nodes in the top and bottom rows. 
The multiplication in $TL_{\delta}(n)$ is given by linear extension of the concatenation
of diagrams, where the concatenated diagram is multiplied by $\delta^r$ if there are $r$ circles appearing in the concatenation
(which are then removed in the resulting concatenated diagram). 
For an example of a concatenation of diagrams see Figure \ref{fig:concattemperleylieb}. 
\end{Definition}

\begin{figure}
\begin{center}
\begin{tikzpicture}
\fill[black!15!white] (0,0) rectangle (3,1);
\foreach \n in {0,1,...,3}
   \foreach \m in {0,1}
       \node at (\n,\m)[circle,fill,inner sep=1.5pt]{};
\draw (0,1) .. controls (0,0.5) and (0,0.5) .. (0,0);
\draw (1,1) .. controls (1,0.5) and (1,0.5) .. (1,0);
\draw (2,1) .. controls (2,0.5) and (2,0.5) .. (2,0);
\draw (3,1) .. controls (3,0.5) and (3,0.5) .. (3,0);
\fill[black!15!white] (4,0) rectangle (7,1);
\foreach \n in {4,5,...,7}
   \foreach \m in {0,1}
       \node at (\n,\m)[circle,fill,inner sep=1.5pt]{};
\draw (4,1) .. controls (4.5,0.6) .. (5,1);
\draw (5,1) .. controls (4.5,0.6) .. (4,1);
\draw (6,1) .. controls (6,0.5) and (6,0.5) .. (6,0);
\draw (7,1) .. controls (7,0.5) and (7,0.5) .. (7,0);
\draw (4,0) .. controls (4.5,0.4) .. (5,0);
\draw (5,0) .. controls (4.5,0.4) .. (4,0);
\fill[black!15!white] (8,0) rectangle (11,1);
\foreach \n in {8,9,...,11}
   \foreach \m in {0,1}
       \node at (\n,\m)[circle,fill,inner sep=1.5pt]{};
\draw (8,1) .. controls (8,0.5) and (8,0.5) .. (8,0);
\draw (9,1) .. controls (9.5,0.6) .. (10,1);
\draw (10,1) .. controls (9.5,0.6) .. (9,1);
\draw (11,1) .. controls (11,0.5) and (11,0.5) .. (11,0);
\draw (9,0) .. controls (9.5,0.4) .. (10,0);
\draw (10,0) .. controls (9.5,0.4) .. (9,0);
\fill[black!15!white] (12,0) rectangle (15,1);
\foreach \n in {12,13,...,15}
   \foreach \m in {0,1}
       \node at (\n,\m)[circle,fill,inner sep=1.5pt]{};
\draw (12,1) .. controls (12,0.5) and (12,0.5) .. (12,0);
\draw (13,1) .. controls (13,0.5) and (13,0.5) .. (13,0);
\draw (14,1) .. controls (14.5,0.6) .. (15,1);
\draw (15,1) .. controls (14.5,0.6) .. (14,1);
\draw (14,0) .. controls (14.5,0.4) .. (15,0);
\draw (15,0) .. controls (14.5,0.4) .. (14,0);
\fill[black!15!white] (0,-2) rectangle (3,-1);
\foreach \n in {0,1,...,3}
   \foreach \m in {-2,-1}
       \node at (\n,\m)[circle,fill,inner sep=1.5pt]{};
\draw (0,-1) .. controls (0.5,-1.5) .. (1,-1);
\draw (1,-1) .. controls (0.5,-1.5) .. (0,-1);
\draw (2,-1) .. controls (2,-1.5) and (0,-1.5) .. (0,-2);
\draw (3,-1) .. controls (3,-1.5) and (3,-1.5) .. (3,-2);
\draw (1,-2) .. controls (1.5,-1.5) .. (2,-2);
\draw (2,-2.) .. controls (1.5,-1.5) .. (1,-2);
\fill[black!15!white] (4,-2) rectangle (7,-1);
\foreach \n in {4,5,...,7}
   \foreach \m in {-2,-1}
       \node at (\n,\m)[circle,fill,inner sep=1.5pt]{};
\draw (4,-1) .. controls (4.5,-1.5) .. (5,-1);
\draw (5,-1) .. controls (4.5,-1.5) .. (4,-1);
\draw (6,-1) .. controls (6.5,-1.5) .. (7,-1);
\draw (7,-1) .. controls (6.5,-1.5) .. (6,-1);
\draw (4,-2) .. controls (4.5,-1.5) .. (5,-2);
\draw (5,-2) .. controls (4.5,-1.5) .. (4,-2);
\draw (6,-2) .. controls (6.5,-1.5) .. (7,-2);
\draw (7,-2) .. controls (6.5,-1.5) .. (6,-2);
\fill[black!15!white] (8,-2) rectangle (11,-1);
\foreach \n in {8,9,...,11}
   \foreach \m in {-2,-1}
       \node at (\n,\m)[circle,fill,inner sep=1.5pt]{};
\draw (8,-1) .. controls (8,-1.5) and (10,-1.5) .. (10,-2);
\draw (9,-1) .. controls (9.5,-1.5) .. (10,-1);
\draw (10,-1) .. controls (9.5,-1.5) .. (9,-1);
\draw (11,-1) .. controls (11,-1.5) and (11,-1.5) .. (11,-2);
\draw (8,-2) .. controls (8.5,-1.5) .. (9,-2);
\draw (9,-2) .. controls (8.5,-1.5) .. (8,-2);
\fill[black!15!white] (12,-2) rectangle (15,-1);
\foreach \n in {12,13,...,15}
   \foreach \m in {-2,-1}
       \node at (\n,\m)[circle,fill,inner sep=1.5pt]{};
\draw (12,-1) .. controls (12,-1.5) and (12,-1.5) .. (12,-2);
\draw (13,-1) .. controls (13.5,-1.5) .. (14,-1);
\draw (14,-1) .. controls (13.5,-1.5) .. (13,-1);
\draw (15,-1) .. controls (15,-1.5) and (13,-1.5) .. (13,-2);
\draw (14,-2) .. controls (14.5,-1.5) .. (15,-2);
\draw (15,-2) .. controls (14.5,-1.5) .. (14,-2);
\fill[black!15!white] (0,-4) rectangle (3,-3);
\foreach \n in {0,1,...,3}
   \foreach \m in {-4,-3}
       \node at (\n,\m)[circle,fill,inner sep=1.5pt]{};
\draw (0,-3) .. controls (0,-3.5) and (0,-3.5) .. (0,-4);
\draw (1,-3) .. controls (1,-3.5) and (3,-3.5) .. (3,-4);
\draw (2,-3) .. controls (2.5,-3.5) .. (3,-3);
\draw (3,-3) .. controls (2.5,-3.5) .. (2,-3);
\draw (1,-4) .. controls (1.5,-3.5) .. (2,-4);
\draw (2,-4) .. controls (1.5,-3.5) .. (1,-4);
\fill[black!15!white] (4,-4) rectangle (7,-3);
\foreach \n in {4,5,...,7}
   \foreach \m in {-4,-3}
       \node at (\n,\m)[circle,fill,inner sep=1.5pt]{};
\draw (4,-3) .. controls (4.5,-3.5) .. (5,-3);
\draw (5,-3) .. controls (4.5,-3.5) .. (4,-3);
\draw (6,-3) .. controls (6,-3.5) and (4,-3.5) .. (4,-4);
\draw (7,-3) .. controls (7,-3.5) and (5,-3.5) .. (5,-4);
\draw (6,-4) .. controls (6.5,-3.5) .. (7,-4);
\draw (7,-4) .. controls (6.5,-3.5) .. (6,-4);
\fill[black!15!white] (8,-4) rectangle (11,-3);
\foreach \n in {8,9,...,11}
   \foreach \m in {-4,-3}
       \node at (\n,\m)[circle,fill,inner sep=1.5pt]{};
\draw (8,-3) .. controls (8.5,-3.5) .. (9,-3);
\draw (9,-3) .. controls (8.5,-3.5) .. (8,-3);
\draw (10,-3) .. controls (10.5,-3.5) .. (11,-3);
\draw (11,-3) .. controls (10.5,-3.5) .. (10,-3);
\draw (8,-4) .. controls (9.5,-3.5) .. (11,-4);
\draw (9,-4) .. controls (9.5,-3.5) .. (10,-4);
\draw (10,-4) .. controls (9.5,-3.5) .. (9,-4);
\draw (11,-4) .. controls (9.5,-3.5) .. (8,-4);
\fill[black!15!white] (12,-4) rectangle (15,-3);
\foreach \n in {12,13,...,15}
   \foreach \m in {-4,-3}
       \node at (\n,\m)[circle,fill,inner sep=1.5pt]{};
\draw (12,-3) .. controls (13.5,-3.5) .. (15,-3);
\draw (13,-3) .. controls (13.5,-3.5) .. (14,-3);
\draw (14,-3) .. controls (13.5,-3.5) .. (13,-3);
\draw (15,-3) .. controls (13.5,-3.5) .. (12,-3);
\draw (12,-4) .. controls (12.5,-3.5) .. (13,-4);
\draw (13,-4) .. controls (12.5,-3.5) .. (12,-4);
\draw (14,-4) .. controls (14.5,-3.5) .. (15,-4);
\draw (15,-4) .. controls (14.5,-3.5) .. (14,-4);
\fill[black!15!white] (0,-6) rectangle (3,-5);
\foreach \n in {0,1,...,3}
   \foreach \m in {-6,-5}
       \node at (\n,\m)[circle,fill,inner sep=1.5pt]{};
\draw (0,-5) .. controls (0,-5.5) and (2,-5.5) .. (2,-6);
\draw (1,-5) .. controls (1,-5.5) and (3,-5.5) .. (3,-6);
\draw (2,-5) .. controls (2.5,-5.5) .. (3,-5);
\draw (3,-5) .. controls (2.5,-5.5) .. (2,-5);
\draw (0,-6) .. controls (0.5,-5.5) .. (1,-6);
\draw (1,-6) .. controls (0.5,-5.5) .. (0,-6);
\fill[black!15!white] (4,-6) rectangle (7,-5);
\foreach \n in {4,5,...,7}
   \foreach \m in {-6,-5}
       \node at (\n,\m)[circle,fill,inner sep=1.5pt]{};
\draw (4,-5) .. controls (5.5,-5.5) .. (7,-5);
\draw (5,-5) .. controls (5.5,-5.5) .. (6,-5);
\draw (6,-5) .. controls (5.5,-5.5) .. (5,-5);
\draw (7,-5) .. controls (5.5,-5.5) .. (4,-5);
\draw (4,-6) .. controls (5.5,-5.5) .. (7,-6);
\draw (5,-6) .. controls (5.5,-5.5) .. (6,-6);
\draw (6,-6) .. controls (5.5,-5.5) .. (5,-6);
\draw (7,-6) .. controls (5.5,-5.5) .. (4,-6);
\end{tikzpicture}
 \caption{All Temperley-Lieb diagrams for $n=4$.}
    \label{fig:temperleyliebn4}
\end{center}
\end{figure}
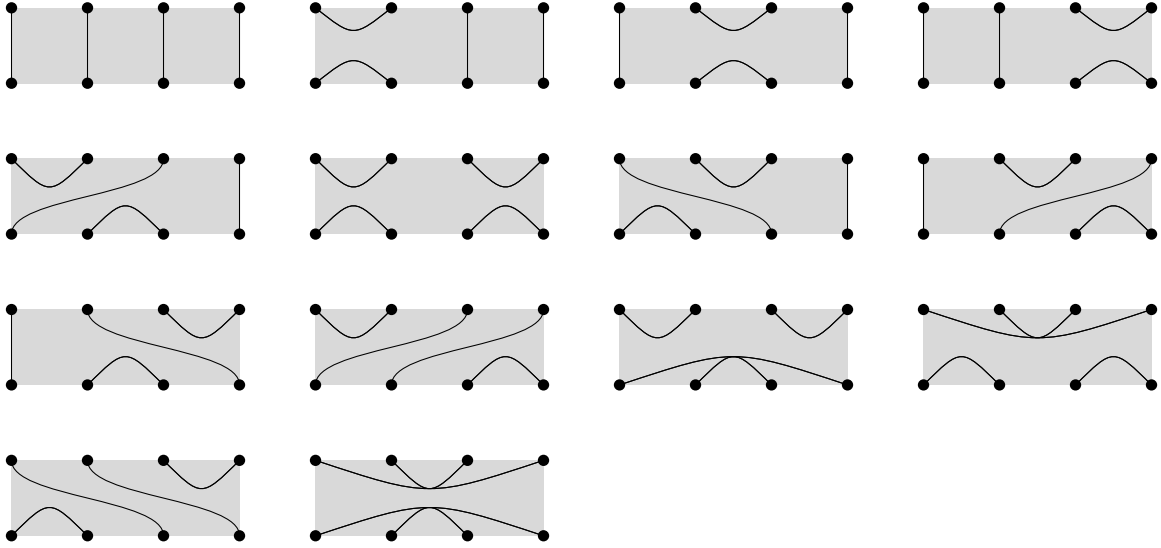

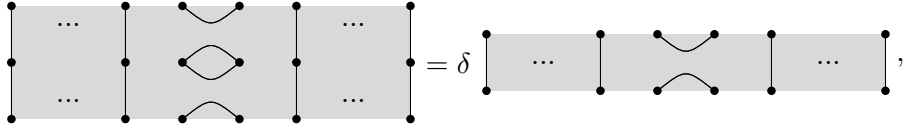
\begin{figure}
\begin{center}
    \begin{tikzpicture}
     \fill[black!15!white] (0,0) rectangle (5.25,0.75);
     \foreach \n in {0,1.5,2.25,3,3.75,5.25}
        \foreach \m in {0,0.75}
            \node at (\n,\m)[circle,fill,inner sep=1.125pt]{};
     \draw (0,0.75) .. controls (0,0.375) and (0,0.375) .. (0,0);
     \draw (1.5,0.75) .. controls (1.5,0.375) and (1.5,0.375) .. (1.5,0);
     \draw (2.25,0.75) .. controls (2.625,0.45) .. (3,0.75);
     \draw (3,0.75) .. controls (2.625,0.45) .. (2.25,0.75);
     \draw (3.75,0.75) .. controls (3.75,0.375) and (3.75,0.375) .. (3.75,0);
     \draw (5.25,0.75) .. controls (5.25,0.375) and (5.25,0.375) .. (5.25,0);
     \draw (2.25,0) .. controls (2.625,0.3) .. (3,0);
     \draw (3,0) .. controls (2.625,0.3) .. (2.25,0);
     \fill[black!15!white] (0,-0.75) rectangle (5.25,0);
     \foreach \n in {0,1.5,2.25,3,3.75,5.25}
        \foreach \m in {-0.75,0}
            \node at (\n,\m)[circle,fill,inner sep=1.125pt]{};
     \draw (0,0) .. controls (0,-0.375) and (0,-0.375) .. (0,-0.75);
     \draw (1.5,0) .. controls (1.5,-0.375) and (1.5,-0.375) .. (1.5,-0.75);
     \draw (2.25,0) .. controls (2.625,-0.3) .. (3,0);
     \draw (3,0) .. controls (2.625,-0.3) .. (2.25,0);
     \draw (3.75,0) .. controls (3.75,-0.375) and (3.75,-0.375) .. (3.75,-0.75);
     \draw (5.25,0) .. controls (5.25,-0.375) and (5.25,-0.375) .. (5.25,-0.75);
     \draw (2.25,-0.75) .. controls (2.625,-0.45) .. (3,-0.75);
     \draw (3,-0.75) .. controls (2.625,-0.45) .. (2.25,-0.75);

          \fill[black!15!white] (6.25,-0.375) rectangle (11.5,0.375);
     \foreach \n in {6.25,7.75,8.5,9.25,10,11.5}
        \foreach \m in {-0.375,0.375}
            \node at (\n,\m)[circle,fill,inner sep=1.125pt]{};
     \draw (6.25,0.375) .. controls (6.25,0) and (6.25,0) .. (6.25,-0.375);
     \draw (7.75,0.375) .. controls (7.75,0) and (7.75,0) .. (7.75,-0.375);
     \draw (8.5,0.375) .. controls (8.875,0.075) .. (9.25,0.375);
     \draw (9.25,0.375) .. controls (8.875,0.075) .. (8.5,0.375);
     \draw (10,0.375) .. controls (10,0) and (10,0) .. (10,-0.375);
     \draw (11.5,0.375) .. controls (11.5,0) and (11.5,0) .. (11.5,-0.375);
     \draw (8.5,-0.375) .. controls (8.875,-0.075) .. (9.25,-0.375);
     \draw (9.25,-0.375) .. controls (8.875,-0.075) .. (8.5,-0.375);
     \node at (5.75,0){$=\delta$};
     \node at (0.75,0.5){...};
     \node at (0.75,-0.5){...};
     \node at (4.5,0.5){...};
     \node at (4.5,-0.5){...};
     \node at (7,0){...};
     \node at (10.75,0){...};
     \node at (11.7,0){,};
    \end{tikzpicture}
    \caption{A concatenation of Temperley-Lieb diagrams in $TL_{\delta}(n)$.}
    \label{fig:concattemperleylieb}
\end{center}
\end{figure}

The representation theory of the Temperley-Lieb algebras has been described in \cite{RSA14}. 
Based on this we can get the structure of a large class of semisimple Temperley-Lieb algebras 
as a second application of our Theorem \ref{thm:Plesken_cellular}.

\begin{Corollary} \label{cor:PleskenTL}
Let $\delta = q+q^{-1}$, where $q\in \mathbb{C}\setminus \{0\}$ is not a root of unity. Then for all $n\in 
\mathbb{N}$ the Temperley-Lieb algebra $TL_{\delta}(n)$ is semisimple and for its Plesken Lie algebra we get
$$\mathcal{L}(TL_{\delta}(n)) \cong \bigoplus_{p=0}^{\lfloor \frac{n}{2}\rfloor} \,\,
\mathfrak{o}\left( {n\choose p}-{n\choose p-1},\mathbb{C}\right).
$$ 
\end{Corollary}

\begin{proof}
All Temperley-Lieb algebras are known to be cellular, where the anti-involution is given by 
swapping top and bottom rows in the diagrams \cite[Theorem 6.7]{GL96}. 
For the specific choice of the parameter $q$ given in the corollary, the Temperley-Lieb algebras $TL_{\delta}(n)$ are semisimple 
\cite[Corollary 4.6]{RSA14}. A complete set of simple modules has been described in \cite[Corollary 4.6]{RSA14}, 
they are indexed by an integer
$p\in \{0,1,\ldots,\lfloor \frac{n}{2}\rfloor\}$, and their
dimensions are given by ${n\choose p}-{n\choose p-1}$ \cite[Equation (2.9)]{RSA14}. Then the claim follows by
applying Theorem \ref{thm:Plesken_cellular}.
\end{proof}


\begin{Remark}
The assumption in Theorem \ref{thm:Plesken_cellular} that the cellular algebra $A$ is semisimple is
necessary. As an example, consider the Temperley-Lieb algebra $TL_0(4)$ over $\mathbb{C}$ for the parameter $\delta=0$
(note that this case is not covered by Corollary \ref{cor:PleskenTL}). The Plesken Lie algebra
$\mathcal{L}(TL_0(4))$  has dimension 4, with basis elements
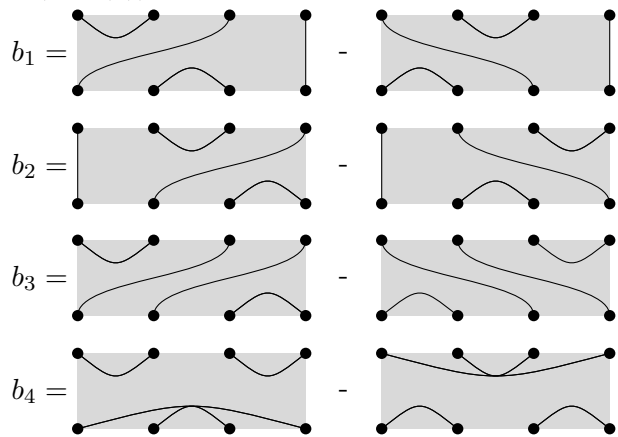
\begin{center}
     \begin{tikzpicture}
        \fill[black!15!white] (0,0) rectangle (3,1);
\foreach \n in {0,1,...,3}
   \foreach \m in {0,1}
       \node at (\n,\m)[circle,fill,inner sep=1.5pt]{};
\draw (0,1) .. controls (0.5,0.6) .. (1,1);
\draw (1,1) .. controls (0.5,0.6) .. (0,1);
\draw (2,1) .. controls (2,0.5) and (0,0.5) .. (0,0);
\draw (3,1) .. controls (3,0.5) and (3,0.5) .. (3,0);
\draw (1,0) .. controls (1.5,0.4) .. (2,0);
\draw (2,0) .. controls (1.5,0.4) .. (1,0);
\node at (3.5,0.5) {-};
\node at (-0.5,0.5){$b_1=$};
\fill[black!15!white] (4,0) rectangle (7,1);
\foreach \n in {4,5,...,7}
   \foreach \m in {0,1}
       \node at (\n,\m)[circle,fill,inner sep=1.5pt]{};
\draw (4,1) .. controls (4,0.5) and (6,0.5) .. (6,0);
\draw (5,1) .. controls (5.5,0.6) .. (6,1);
\draw (6,1) .. controls (5.5,0.6) .. (5,1);
\draw (7,1) .. controls (7,0.5) and (7,0.5) .. (7,0);
\draw (4,0) .. controls (4.5,0.4) .. (5,0);
\draw (5,0) .. controls (4.5,0.4) .. (4,0);
\fill[black!15!white] (0,-1.5) rectangle (3,-0.5);
\foreach \n in {0,1,...,3}
   \foreach \m in {-1.5,-0.5}
       \node at (\n,\m)[circle,fill,inner sep=1.5pt]{};
\draw (0,-0.5) .. controls (0,-1) and (0,-1) .. (0,-1.5);
\draw (1,-0.5) .. controls (1.5,-0.9) .. (2,-0.5);
\draw (2,-0.5) .. controls (1.5,-0.9) .. (1,-0.5);
\draw (3,-0.5) .. controls (3,-1) and (1,-1) .. (1,-1.5);
\draw (2,-1.5) .. controls (2.5,-1.1) .. (3,-1.5);
\draw (3,-1.5) .. controls (2.5,-1.1) .. (2,-1.5);
\node at (3.5,-1) {-};
\node at (-0.5,-1){$b_2=$};
\fill[black!15!white] (4,-1.5) rectangle (7,-0.5);
\foreach \n in {4,5,...,7}
   \foreach \m in {-1.5,-0.5}
       \node at (\n,\m)[circle,fill,inner sep=1.5pt]{};
\draw (4,-0.5) .. controls (4,-1) and (4,-1) .. (4,-1.5);
\draw (5,-0.5) .. controls (5,-1) and (7,-1) .. (7,-1.5);
\draw (6,-0.5) .. controls (6.5,-0.9) .. (7,-0.5);
\draw (7,-0.5) .. controls (6.5,-0.9) .. (6,-0.5);
\draw (5,-1.5) .. controls (5.5,-1.1) .. (6,-1.5);
\draw (6,-1.5) .. controls (5.5,-1.1) .. (5,-1.5);
\end{tikzpicture}
\end{center}
~\vskip0.3cm 

\begin{center}
\begin{tikzpicture}
\fill[black!15!white] (0,-3) rectangle (3,-2);
\foreach \n in {0,1,...,3}
   \foreach \m in {-3,-2}
       \node at (\n,\m)[circle,fill,inner sep=1.5pt]{};
\draw (0,-2) .. controls (0.5,-2.4) .. (1,-2);
\draw (1,-2) .. controls (0.5,-2.4) .. (0,-2);
\draw (2,-2) .. controls (2,-2.5) and (0,-2.5) .. (0,-3);
\draw (3,-2) .. controls (3,-2.5) and (1,-2.5) .. (1,-3);
\draw (2,-3) .. controls (2.5,-2.6) .. (3,-3);
\draw (3,-3) .. controls (2.5,-2.6) .. (2,-3);
\node at (3.5,-2.5) {-};
\node at (-0.5,-2.5){$b_3=$};
\fill[black!15!white] (4,-3) rectangle (7,-2);
\foreach \n in {4,5,...,7}
   \foreach \m in {-3,-2}
       \node at (\n,\m)[circle,fill,inner sep=1.5pt]{};
\draw (4,-2) .. controls (4,-2.5) and (6,-2.5) .. (6,-3);
\draw (5,-2) .. controls (5,-2.5) and (7,-2.5) .. (7,-3);
\draw (6,-2) .. controls (6.5,-2.4) .. (7,-2);
\draw (4,-3) .. controls (4.5,-2.6) .. (5,-3);
\fill[black!15!white] (0,-4.5) rectangle (3,-3.5);
\foreach \n in {0,1,...,3}
   \foreach \m in {-4.5,-3.5}
       \node at (\n,\m)[circle,fill,inner sep=1.5pt]{};
\draw (0,-3.5) .. controls (0.5,-3.9) .. (1,-3.5);
\draw (1,-3.5) .. controls (0.5,-3.9) .. (0,-3.5);
\draw (2,-3.5) .. controls (2.5,-3.9) .. (3,-3.5);
\draw (3,-3.5) .. controls (2.5,-3.9) .. (2,-3.5);
\draw (0,-4.5) .. controls (1.5,-4.1) .. (3,-4.5);
\draw (1,-4.5) .. controls (1.5,-4.1) .. (2,-4.5);
\draw (2,-4.5) .. controls (1.5,-4.1) .. (1,-4.5);
\draw (3,-4.5) .. controls (1.5,-4.1) .. (0,-4.5);
\node at (3.5,-4) {-};
\node at (-0.5,-4){$b_4=$};
\fill[black!15!white] (4,-4.5) rectangle (7,-3.5);
\foreach \n in {4,5,...,7}
   \foreach \m in {-4.5,-3.5}
       \node at (\n,\m)[circle,fill,inner sep=1.5pt]{};
\draw (4,-3.5) .. controls (5.5,-3.9) .. (7,-3.5);
\draw (5,-3.5) .. controls (5.5,-3.9) .. (6,-3.5);
\draw (6,-3.5) .. controls (5.5,-3.9) .. (5,-3.5);
\draw (7,-3.5) .. controls (5.5,-3.9) .. (4,-3.5);
\draw (4,-4.5) .. controls (4.5,-4.1) .. (5,-4.5);
\draw (5,-4.5) .. controls (4.5,-4.1) .. (4,-4.5);
\draw (6,-4.5) .. controls (6.5,-4.1) .. (7,-4.5);
\draw (7,-4.5) .. controls (6.5,-4.1) .. (6,-4.5);   
    \end{tikzpicture}
\end{center}   

A lengthy, but straightforward calculation by concatenating diagrams yields the following Lie bracket multiplication 
table for the Plesken Lie algebra. 

\begin{center}
\begin{tabular}{|c||c|c|c|c|}
\hline 
$[\,.\,,\,.\,]$ & $b_1$ & $b_2$ & $b_3$ & $b_4$ \\
 \hline\hline
 $b_1$ & $0$ & $-b_1-b_2$ & $b_3-b_4$ & $0$ \\ \hline
 $b_2$ & $b_1+b_2$ & $0$ & $-b_3-b_4$ & $0$ \\ \hline
 $b_3$ & $-b_3+b_4$ & $b_3+b_4$ & $0$ & $0$ \\ \hline
 $b_4$ & $0$ & $0$ & $0$ & $0$ \\
 \hline
 \end{tabular}
\end{center}
\medskip

Recall that the derived series of a Lie algebra is obtained by repeatedly taking commutators. 
From the multiplication table we 
determine the derived subalgebras of $\mathcal{L}(TL_0(4))$ as
\begin{eqnarray*}
(\mathcal{L}(TL_0(4)))' & = & \mathrm{span}(b_1+b_2,b_3,b_4) \\
(\mathcal{L}(TL_0(4)))'' & = & \mathrm{span}(b_4) \\
(\mathcal{L}(TL_0(4)))''' & = & 0 
 \end{eqnarray*}
Thus, the Plesken Lie algebra of $TL_0(4)$ is a solvable Lie algebra of derived length 3. This implies 
that it can not be a direct sum of orthogonal Lie algebras as in Theorem \ref{thm:Plesken_cellular}
(since $\mathfrak{o}(r,\mathbb{C})' =\mathfrak{o}(r,\mathbb{C})$ for $r\ge 3$ and 
  $\mathfrak{o}(r,\mathbb{C})'=0$ for $r=1,2$, so if a direct sum of orthogonal Lie algebras
  is solvable then of derived length 1). In particular, our calculations together with Theorem \ref{thm:Plesken_cellular}
 also show that the 
  cellular algebra $TL_0(4)$ is not semisimple. 
\end{Remark}



\begin{thebibliography}{19}

\bibitem{CT07}
A.\,M.\,Cohen, D.\,E.\,Taylor,
{\em On a certain Lie algebra defined by a finite group},
Amer. Math. Monthly 114 (2007), 633-639.

\bibitem{CLX23}
W.\,Cui, L.\,Luo, Z.\,Xu,
{\em Asymptotic Schur algebras and cellularity of $q$-Schur algebras},
Preprint (2023), arXiv:2305.14633.

\bibitem{EH18}
K.\,Erdmann, T.\,Holm,
{\em Algebras and Representation Theory},
Springer Undergraduate Mathematics Series (SUMS), Springer (2018).

\bibitem{FHH09}
D.\,Flath, T.\,Halverson, K.\,Herbig,
{\em The planar rook algebra and Pascal's triangle},
L'Enseignement Math\'ematique 55 (2009), 77-92.

\bibitem{G07}
M.\,Geck,
{\em Hecke algebras of finite type are cellular},
Invent. Math. 169 (2007), 501-517.

\bibitem{GL96}
J.\,J.\,Graham, G.\,I.\,Lehrer,
{\em Cellular algebras},
Invent. Math. 123 (1996), 1-34.

\bibitem{Jo94}
V.\,F.\,R.\,Jones,
{\em The Potts model and the symmetric group},
In: {\em Subfactors: Proceedings of the Taniguchi Symposium on Operator Algebras 
(Kyuzeso 1993)}, World Sci. Publishing, River Edge, NJ, 1994, 259-267.

\bibitem{KX98}
S.\,K\"onig, C.\,C.\,Xi,
{\em On the structure of cellular algebras},
Canadian Math. Soc. Conference Proceedings, vol. 24 (1998), 365-385.

\bibitem{M10}
I.\,Marin,
{\em Group algebras of finite groups as Lie algebras},
Comm. Algebra 38 (2010), 2572-2584.

\bibitem{Ma94}
P.\,Martin,
{\em Temperley-Lieb algebras for nonplanar statistical mechanics - the partition algebra
construction},
J. Knot Theory Ramifications 3 (1994), 51-82.

\bibitem{MT23}
A.\,Mathas, D.\,Tubbenhauer,
{\em Cellularity of KLR and weighetd KLRW algebras via crystals},
Preprint (2023), arXiv:2309.13867. 

\bibitem{RSA14}
D.\,Ridout, Y.\,Saint-Aubin,
{\em Standard modules, induction and the structure of the Temperley-Lieb algebra},
Adv. Theor. Math. Phys. 18.5 (2014), 957-1041.

\bibitem{W25}
N.\,Wirries,
{\em Plesken Lie algebras of associative algebras with involution},
Master thesis, Leibniz Universit\"at Hannover, 2025.

\bibitem{Xi99}
C.\,Xi,
{\em Partition algebras are cellular},
Compos. Math. 119 (1999), 107-118.


\end{thebibliography}
\end{document}